\documentclass[reqno,12pt]{amsart}
\usepackage{amsmath}
\usepackage{amssymb}
\usepackage{amscd}
\input{epsf}

 \DeclareMathOperator{\Z}{\mathbb
Z} \DeclareMathOperator{\C}{\mathbb C}
\DeclareMathOperator{\Pn}{\mathbb P}
\DeclareMathOperator{\alphabar}{\underline{\alpha}}
\DeclareMathOperator{\N}{\mathbb N}

\newtheorem{fact}{Fact}[section]
\newtheorem{lemma}[fact]{Lemma}
\newtheorem{theorem}[fact]{Theorem}
\newtheorem{definition}[fact]{Definition}
\newtheorem{example}[fact]{Example}
\newtheorem{rremark}[fact]{Remark}

\newenvironment{remark}{\begin{rremark} \rm}{\end{rremark}}

\DeclareMathOperator{\E}{\mathcal E} \DeclareMathOperator{\tp}{Tp}
\DeclareMathOperator{\codim}{codim} \DeclareMathOperator{\ts}{Ts}
\def\F{\mathbf F}
\def\G{\mathbf G}
\def\P{\mathbb P}
\def\B{\mathbf B}
\DeclareMathOperator{\Gr}{Gr} \DeclareMathOperator{\Q}{\mathbb Q}

\title{The general quadruple point formula}

\author{R. Marangell}
\address{Department of Mathematics, University of North Carolina at Chapel Hill, USA}
\email{roble81@email.unc.edu}

\author{R. Rim\'anyi}
\address{Department of Mathematics, University of North Carolina at Chapel Hill, USA}
\email{rimanyi@email.unc.edu}

\begin{document}

\thanks{\noindent Supported by NSF grants DMS-0405723 (2nd author)\\
AMS Subject classification 14N10, 57R45}

\begin{abstract} Maps between manifolds $M^m\to N^{m+\ell}$ ($\ell>0$) have
multiple points, and more generally, multisingularities. The closure
of the set of points where the map has a particular multisingularity
is called the multisingularity locus. There are universal relations
among the cohomology classes represented by multisingularity loci,
and the characteristic classes of the manifolds. These relations
include the celebrated Thom polynomials of monosingularities. For
multisingularities, however, only the form of these relations is
clear in general (due to Kazarian \cite{kazamulti}), the concrete
polynomials occurring in the relations are much less known. In the
present paper we prove the first general such relation outside the
region of Morin-maps: the general quadruple point formula. We apply
this formula in enumerative geometry by computing the number of
4-secant linear spaces to smooth projective varieties. Some other
multisingularity formulas are also studied, namely 5, 6, 7 tuple
point formulas, and one corresponding to $\Sigma^2\Sigma^0$
multisingularities.
\end{abstract}

\maketitle

\section{Introduction}
\label{intro}

Let $f:M^m\to N^n$ be a holomorphic map between compact complex
manifolds, and let $\ell=n-m> 0$. Associated with a list of
singularities $\alphabar=(\alpha_1,\ldots,\alpha_r)$ one can
consider the following multisingularity locus in the source manifold
$M$: the collection of points $x$ where the map has singularity
$\alpha_1$, and $f(x)$ has another $r-1$ preimages
$\{x_2,x_3,...,x_r \}$, with $f$ having singularities $\alpha_i$ at
$x_i$. We will be concerned with the cohomology class
$m_{\alphabar}\in H^*(M)$ represented by the closure of the
multisingularity locus, and its image $n_{\alphabar}\in
H^*(N)=H^*(N;\Q)$ under the Gysin homomorphism. In the whole paper
cohomology is meant with rational coefficients.

The cohomology classes $m_{\alphabar}$, $n_{\alphabar}$ satisfy
universal identities.  The word ``universal'' means that the
dependence of these identities on the manifolds $M$, $N$, and the
map $f$, is only via the characteristic classes $c(TM)$, and
$f^*(c(TN))$. We mention two such prototype formulas. First,
$m_{A_1}=c_{\ell+1}(f)$, that is the cohomology class represented by
the points where the map $f$ is singular is equal to the $\ell+1$'st
Chern class of the map, where $c(f)=c(f^*TM-TN)$. Another one is
$m_{A_0^2}=f^*(n_{A_0})-c_{\ell}(f)$. This identity expresses the
cohomology class represented by the double point locus, in terms of
the cohomology class $n_{A_0}$  represented by the image of $f$, and
the $\ell$th Chern class of $f$ (we will define precisely what the
singularities $A_1$ and $A_0^2$ are in Chapter \ref{KazaTh}).

The universal identities among multisingularity classes have
applications in differential topology and algebraic geometry. In
differential topology these identities can be used to show that a
certain multisingularity locus is not zero, provided only that
certain characteristic classes are nonzero. Hence certain
degenerations of the map are {\em forced} by the global topology of
the source and target spaces. The polynomials expressing
monosingularity classes (called Thom polynomials) are in close
relation with polynomials useful in geometry: e.g. relations in
presentations of cohomology rings of moduli spaces \cite{forms}, or
polynomials governing the combinatorics of Schubert calculus
\cite{lfrr02}, \cite{lfrr03}.

In the second half of the 20th century the main application of
multisingularity class identities were in enumerative geometry. When
we want to count certain geometric objects satisfying certain
properties, we can often encode the problem by setting up a map,
whose certain multisingularity locus is in bijection with the
counted objects. Characteristic classes of maps are usually easy to
handle, so using the universal identities, we have a formula for the
counted objects. This approach was used by e.g. Kleiman
\cite{kleiman}, Katz \cite{katz}, Colley \cite{colley}, and recently
by Kazarian  \cite{kazamorin}, \cite{kazamulti} with great success.
One advantage of this method is that it often avoids the problem of
excess intersection.

The limitation of this method is the fact that hardly any {\em
general} multisingularity identities are known. By ``general" we
mean a formula which is valid for all dimensional settings, and for
all maps with expected multisingularities. For example the two
prototype formulas above are valid for any dimension $m< n$, but
e.g. $m_{A_2}=c_1^2(f)+c_2(f)$ is only valid for maps with $\ell=0$.
Some `general' monosingularity formulas are known, and only three
general multisingularity formulas. These are the double-point
formula, the triple-point formula, and a formula concerning $A_0A_1$
multisingularities. The reason why higher formulas are considerably
harder is roughly speaking that higher multisingularities do not
only interfere with the simplest monosingularities of maps, the
so-called Morin singularities (a.k.a. corank-one, or curvilinear
singularities). The quadruple points of a map $M^m\to N^{m+\ell}$
have codimension $3\ell$ in the source, while the codimension of
non-Morin singularities is $2\ell+4$. Hence the non-Morin ones
interfere with the quadruple points.

Methods of algebraic geometry have been applied to find several
multiple point formulas that are valid for maps with only Morin
singularities, see works of Kleiman, Kazarian \cite{kleiman2},
\cite{kazamorin} and references therein. These formulas are not
general either, because they agree with the general formulas only up
to characteristic classes supported on non-Morin singularity loci.

The main result of the present paper is proving the general
$i$-tuple point formula (\ref{general-main}) for $i\leq7$. Our
method has three pillars, as follows: (i) Kazarian found the general
form of multisingularity formulas \cite{kazamulti}; (ii)
Rim\'anyi---based on Sz\H ucs's construction of classifying space of
multisingularities in \cite{rrsz98}---found an ``interpolation''
method to gather information on the polynomials governing the
multisigularity identities \cite{rrmulipt}; (iii) recently B\'erczi
and Szenes in \cite{bsz06} used advanced localizations to calculate
the Thom polynomial of the $A_i$ singularity for $i\leq 6$ (the
$A_3$ formula was announced in \cite{a3}). What we will show using
interpolation is that after certain identifications, the residue
polynomial of e.g. quadruple points is equal to the Thom polynomial
of $A_3$ singularities---for a different dimension setting.

Since the interpolation method is topological in nature, our method
is topological. In Section~\ref{4-secants} we will show an
enumerative geometry application. Namely, we will calculate the
number (the cohomology class) of 4-secant linear spaces to a smooth
projective variety. For a smooth surface in 10 dimensional
projective space we recover the Hilbert scheme calculation of
\cite{lehn99}, for the other dimension settings our result is new.

The advance of this paper, that is, the step from triple point
formulas to higher multiple point formulas can be compared with a
recent advance in algebraic geometry: from the study of the space of
triangles \cite{triangle} to the study of the space of tetrahedra
\cite{tetrahedron}.

Our approach to reduce multisingularity polynomials to some other,
easier and known ones has two limitations. First, not many general
Thom polynomials (called Thom series) are known; the best results
seem to be \cite{bsz06}, \cite{lfrrpp}. The other limitation is that
the direct interpolation method breaks down where moduli of
singularities occur, that is in $\codim\geq 6\ell+9$.

\smallskip

The authors are grateful to A. Buch, J. Damon, L. Feh\'er, and A.
Szenes for useful discussions.

\section{Singularities and multisingularities}
\label{KazaTh}

\subsection{Contact singularities of maps}
\label{Sings} Fix integers $m < n $ and let $\ell = n-m$. We say
that maps or map germs mapping from an $m$ dimensional space to an
$n$ dimensional space have {\em relative dimension} $\ell$.

Consider $\mathcal{E}(m,n)$, the vector space of holomorphic map
germs $(\C^m,0)\rightarrow \C^n$. The subspace consisting of germs
$(\C^m,0)\to (\C^n,0)$ is denoted by $\E^0(m,n)$. The vector space
$\mathcal{E}(m):=\mathcal{E}(m,1)$ is a local algebra with maximal
ideal $\E^0(m)$. The space $\mathcal{E}(m,n)$ is a module over
$\mathcal{E}(m)$, with $\E^0(m,n)$ a submodule. A map $f\in
\E^0(m,n)$ induces a pullback $f^*:\E(n)\to \E(m)$ by composition.

\begin{definition}
\label{localalglebra} The local algebra $Q_f$ of a germ $f \in
\E^0(m,n)$ is defined by $Q_f = \E(m)/ (f^{*}\E^0(n))$.
\end{definition}

We will be concerned with germs $f$ for which the local algebra is
finite dimensional; i.e. the ideal $(f^{*}\E^0(n))$ contains a power
of the maximal ideal. We call these germs {\em finite}. For such a
germ, in local coordinates, $f=(f_1(x_1, \ldots, x_m), \cdots,
f_n(x_1, \ldots, x_m))$, we have $Q_f = \C[[x_1, \ldots, x_m]] /
(f_1, \ldots, f_n)$.

\begin{definition} \label{contactequivalent}
For two germs, $f$ and $g$, we say that $f$ is {\em contact
equivalent} to $g$, if $Q_f \cong Q_g$; that is, their local
algebras are isomorphic. An equivalence class $\eta\subset
\E^0(m,n)$ will be called a (contact) singularity.
\end{definition}

In singularity theory one considers the so called contact group
$\mathcal{K}(m,n)$ acting on the vector space $\E^0(m,n)$, and
defines germs to be contact equivalent if they are in the same
orbit. It is a theorem of Mather \cite{mather6} that for finite germs
the two definitions are equivalent. Thus, for the rest of this paper
all singularities to which we refer will be finite in the previous sense.

\begin{remark} \label{remA} \rm
The group $\mathcal{K}(m,n)$ contains the group of holomorphic
reparametrizations of the source $(\C^m,0)$ and target spaces
$(\C^n,0)$. Hence for a map $f:M^m\to N^n$ between manifolds it
makes sense to talk about the contact singularity of $f$ at a point
in $M$. Hence, for a map $f:M^m\to N^n$ and a singularity
$\eta\subset \E^0(m,n)$, we can define the singularity subset
$$\eta(f)=\{x\in M| \text{ the germ of $f$ at $x$ belongs to }
\eta\}.$$
\end{remark}

\subsection{The zoo of singularities}
\label{zoo}
The classification of finite singularities is roughly the same as
the classification of finite dimensional commutative local
$\C$-algebras. Only `roughly', because for a given $m$ and $n$ only
algebras that can be presented by $m$ generators and $n$ relations
turn up as local algebras of singularities $\C^m,0\to \C^n,0$.

A natural approach is to try to classify singularities in the order
of their codimensions in $\E^0(m,n)$. For large $\ell$ the
classification of small codimensional singularities is as follows
(see e.g.~\cite{agvl}).

\begin{center}\begin{tabular}{l || c || c || c | c || c | c | c ||}
$\codim$ & 0 & $\ell+1$ & $2\ell+2$ & $2\ell +4$ & $3\ell+3$ &
$3\ell +4$ & $3\ell +5$ \\ \hline\hline
$\Sigma^0$ & $A_0$ & & & & & & \\
$\Sigma^1$ & & $A_1$ & $A_2$ & & $A_3$ & & \\
$\Sigma^2$ & & & & $III_{2,2}$ & & $I_{2,2}$ & $III_{2,3}$\\
\end{tabular}\end{center}
Here we use the following notations: $A_i$ means the singularity
with local algebra $\C[x]/(x^{i+1})$; $I_{a,b}$ means the
singularity with local algebra $\C[x,y]/(xy,x^a+y^b)$; and
$III_{a,b}$ means the singularity with local algebra
$\C[x,y]/(x^a,xy,y^b)$. The symbol $\Sigma^r$ is a property of a
singularity, it means that the derivative drops rank by $r$,
equivalently, that the local algebra can be minimally generated by
$r$ generators. The $\Sigma^{\leq 1}$ singularities are called Morin
singularities (a.k.a. corank 1, or curvilinear singularities). As
one studies singularities of high codimension, they appear in
moduli. However for the main result of the present paper we can
avoid working with them.

Observe that we gave the classification independent of $m$ and $n$,
that is, we gave the same name for singularities for different
dimension settings. E.g. the following are all $A_2$ germs:
$x\mapsto x^3$ ($m=n=1$), $(x,y)\mapsto (x^3,y)$ ($m=2,n=2$),
$(x,y)\mapsto (x^3+xy,y)$ ($m=n=2$). An essential difference between
the latter two is that the last one is {\em stable} (it is called
cusp singularity), the other one is not (stable representatives will
play an important role in Section~\ref{interpolation}).

As we already noted in Remark \ref{remA}, singularity submanifolds
can stratify the source space of a map $f:M^m\to N^n$ between
manifolds. We want to study a finer stratification though---one
which corresponds to multisingularities.

\subsection{Multisingularities}\label{multi}

Consider contact singularities $\alpha_i\subset \E^0(m,n)$, $m>n$.

\begin{definition} \label{multisingularity}
A {\em multisingularity} $\underline{\alpha}$ is a multi-set of
singularities $(\alpha_1, \ldots, \alpha_r)$  together with a
distinguished element, denoted $\alpha_1$.
\end{definition}
\noindent For reasons explained in Section \ref{admis}, we define
the codimension of a multisingularity $(\alpha_1, \ldots, \alpha_r)$
by
\begin{equation}\label{codimmulti}
\codim \alphabar=(r-1)\ell+\sum \codim \alpha_i. \end{equation}
Hence the codimension does not depend on the order of the
monosingularities. The list of multisingularities of small
codimension (when $\ell$ is large) is given in the following table.

\smallskip

\begin{center}\begin{tabular}{l || c || c | c || c | c | c | c || c | c| c| c| c|}
$\codim$ & 0 & $\ell$ & $\ell$+1 & 2$\ell$ & 2$\ell$+1 & 2$\ell$+2 &
2$\ell+4$ & 3$\ell$ & $3\ell+1$ & $3\ell+2$ & $3\ell+3$ & $3\ell+4$
\\ \hline\hline
$\Sigma^0$ & $A_0$ & $A_0^2$ & & $A_0^3$ & & & & $A_0^4$ & & & &\\
\hline
$\Sigma^1$ & & & $A_1$ &  & $A_1A_0$ & $A_2$& & & $A_1A_0^2$& $A_2A_0$& $A_3$& \\
 & & &  &  &  & & & &  &$A_1^2$ & &\\
\hline $\Sigma^2$ & & & & & & & $III_{2,2}$ & & & & & $III_{2,2}A_0$ \\
  & & & & & & & & & & & & $I_{2,2}$
\\ \hline
\end{tabular}\end{center}

\noindent Here we used the notation $\alpha_1\alpha_2\ldots$ for the multiset
$(\alpha_1,\alpha_2,\ldots)$, and any of its permutations.

\begin{definition}
\label{multilocussource} Let $f:M^m \to N^n$ be a holomorphic map of
complex manifolds, and $\alphabar = (\alpha_1, \alpha_2, \ldots,
\alpha_r)$ a multisingularity. We define the following
multisingularity loci in $M$ and $N$
$$M_{\underline{\alpha}} =
\{x_1 \in M | f(x_1) \text{ has exactly } r \text{ pre-images } x_1,
\ldots, x_r, \text{ and } f \text{ has singularity } \alpha_i \text{
at } x_i\},$$ and $N_{\alphabar}= f(M_{\alphabar}).$
\end{definition}
If we permute the monosingularities in $\alphabar$, i.e. choose
another singularity to be $\alpha_1$, then $M_{\alphabar}$ changes,
while $N_{\alphabar}$ does not.

\subsection{Admissible maps} \label{admis} The main point of the present paper is
to study certain identities among cohomology classes represented by
multisingularity loci. We can only expect such identities if the map
satisfies certain transversality conditions. We will define these
maps in the present section, and call them `admissible'.

As before, the codimension of the singularity $\eta$ in
$\E^0(m,n)$ is denoted by $\codim \eta$. It is reasonable to
expect that for a `nice enough' map $f:M\to N$, the codimension of
$\eta(f)$ in $M$ is the same (see Remark \ref{remA}). Indeed, for
a map $f:M^m\to N^n$ one can consider the bundle
$$\{\text{germs } (M,x)\to (N,f(x))\}\to \{(x,f(x))|x\in M\}$$ together with the section
$(x,f(x)) \mapsto $ the germ of $f$ at $x$. (Precisely speaking, one
should consider jet approximations to have a finite rank bundle.)
The fibers of this bundle are identified with $\E^0(m,n)$, so we can
consider $\eta$ in each. Thus we obtain a submanifold of codimension
$\codim \eta$ in the total space. The set $\eta(f)$ is the preimage
of this submanifold along the section. Hence, for transversal
sections the codimension of $\eta(f)$ in $M$ is $\codim \eta$. An
`admissible-for-monosingularities' map must have this transversality
property. If we worked over the real numbers we would have the
transversality theorem guaranteeing that almost all maps are
admissible-for-monosingularities.

We need, however, the admissibility property for multisingularities
as well. This more sophisticated notion uses the classifying space
of multisingularities, see \cite{rrsz98}, \cite{kazamulti} as
follows. A map $f:M\to N$ induces a map $k_f$ from $N$ to a space
$X$ called the {\em classifying space of multisingularities}. The
infinite dimensional space $X$ has a finite codimensional
submanifold $X_{\alphabar}$ corresponding to the multisingularity
$\alphabar$. The set $N_{\alphabar}$ is the preimage of
$X_{\alphabar}$ along the map $k_f$. The map $f$ is defined to be
admissible if $k_f$ is transversal to $X_{\alphabar}$ for all
$\alphabar$. Over the real numbers almost all maps are admissible.

\begin{remark} Another way to define admissibility is to require
that a natural section of the ``multijet bundle'' is transversal to
certain submanifolds in the total space, as in the Multijet
Transversality Theorem, see \cite[Thm. 4.13]{gg}. Either way,
admissible maps are admissible-for-monosingularities, together with
the property that the closures of the $f$-images of the
monosingularity submanifolds $\alpha(f)$ satisfy certain
transversality properties. In Figure 1 both maps are
admissible-for-monosingula\-ri\-ties, but only the first map is
admissible for multisingularities.
\end{remark}

\begin{figure}
\epsffile{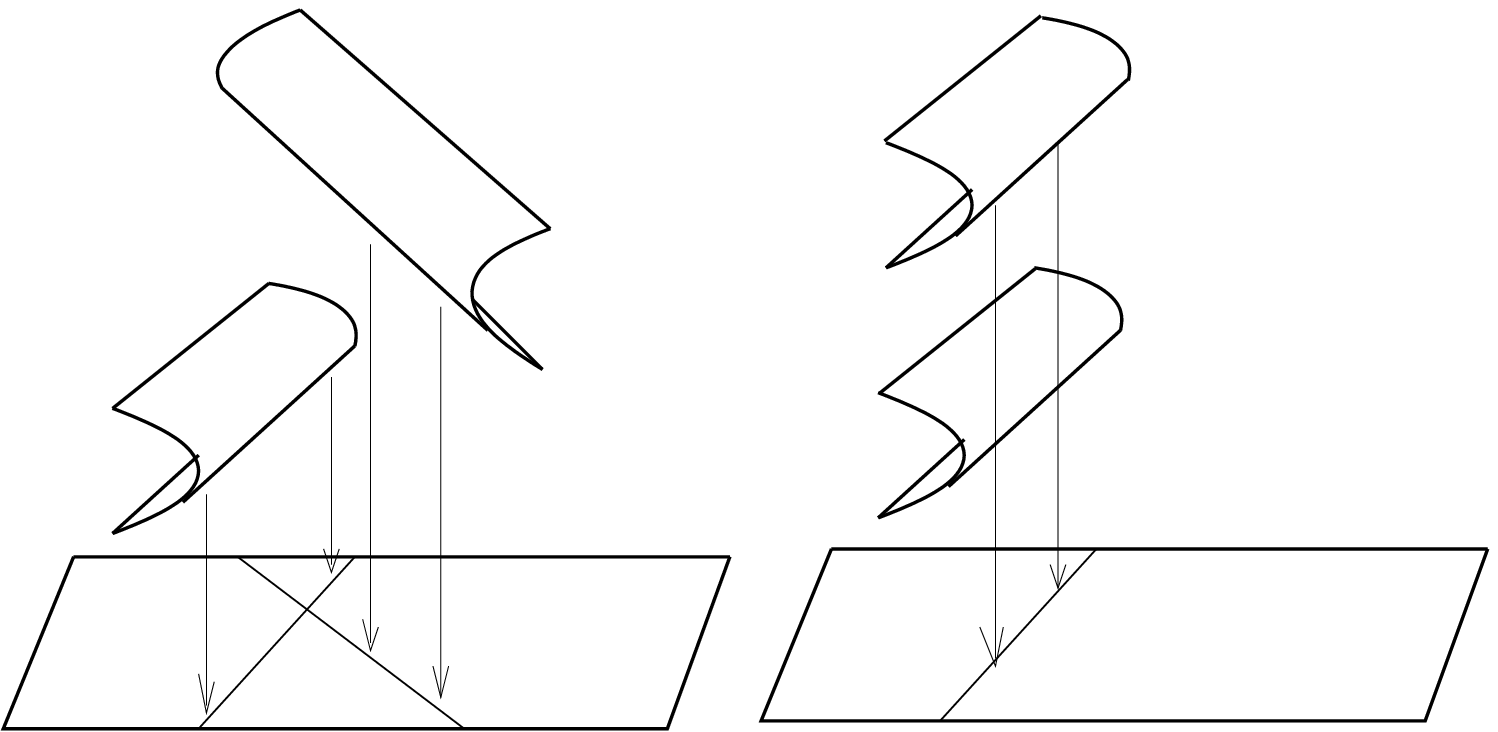} \caption{ } \label{figure1}
\end{figure}

The codimension of $X_{\alphabar}$ in $X$ is $r\ell+\sum \codim
\alpha_i$. Therefore, for an admissible map, the
codimension of $N_{\alphabar}$ in $N$ is $r\ell+\sum \codim
\alpha_i$;
 and the codimension of $M_{\alphabar}$ in $M$ is
$(r-1)\ell+\sum \codim \alpha_i$, (\textit{cf.} formula
(\ref{codimmulti})).  It also follows that for
admissible maps, the closures of the loci $M_{\alphabar}$ and
$N_{\alphabar}$ support fundamental homology classes. We will call
the Poincar\'e duals of these classes the cohomology classes
represented by the $\alphabar$ multisingularity loci in the source and
target manifolds.

\subsection{Cohomology classes represented by multisingularity submanifolds}
Let $f:$ $M^m$ $\to$ $N^n$ be an admissible map. Denote by
$\overline{m}_{\alphabar} = \left[\bar{M _{\alphabar}}\right] \in
H^{\codim \alphabar}(M)$ the cohomology class represented by the
closure of the $\alpha$-multisingularity locus in the source, and
$\overline{n}_{\alphabar} = \left[\bar{N _{\alphabar}}\right] \in
H^{\ell+\codim \alphabar}(N)$, in the target. Since it is often of
use to consider these classes $\overline{m}_{\alphabar}$,
$\overline{n}_{\alphabar}$ with their natural multiplicities, we let
 \begin{center} $m_{\alphabar} =
\#Aut(\alpha_2, \ldots, \alpha_r)\overline{m}_{\alphabar}$,\qquad
$n_{\alphabar} = \#Aut(\alpha_1, \alpha_2, \ldots,
\alpha_r)\overline{n}_{\alphabar}$, \end{center} where
$\#Aut(\alpha_1, \alpha_2, \ldots, \alpha_r) = \#Aut(\alphabar)$ is
the number of permutations $\sigma \in \mathfrak{S}_r$ such that
$\alpha_{\sigma(i)}=\alpha_i$ for all $i$ from 1 to $r$. So if
$\alphabar$ contains $k_1$ singularities of type $\alpha_1$, $k_2$
of type $\alpha_2$, etc., then $\#Aut(\alphabar)= k_1!k_2!\ldots$.

The degree of the restriction map $f:M_{\alphabar} \rightarrow
N_{\alphabar}$ is the number of $\alpha_1$ singularities in
$\alphabar$, hence we have the following relation
\begin{equation}
\label{nalpha}
 n_{\alphabar} = f_!(m_{\alphabar}),
\end{equation}
where $f_!$ is the Gysin homomorphism.

\subsection{Classes of multisingularity loci in terms of
characteristic classes}

The virtual normal bundle $\nu(f)$ of a map $f:M^m\to N^n$ is the
formal difference $f^*(TN)-TM$ of bundles over $M$. This is an actual bundle
if $f$ is an immersion. The total Chern class of a map is defined to be the total
Chern class of its virtual normal bundle, $$c(f)=c(f^*(TN)-TM) =
\frac{c(f^*(TN))}{c(TM)}= \frac{f^*(c(N))}{c(M)}.$$

A classical theorem of Thom \cite{thom56} is that {\em
mono}singularity loci in the source can be expressed as a polynomial
(the Thom polynomial) of the Chern classes of the map. The
generalization for multisingularity loci was found by Kazarian.

\begin{theorem}[Kazarian \cite{kazamulti}] For multisingularities
$\alphabar = (\alpha_1,\ldots , \alpha_r)$ and $J \subset \{1, \ldots, r\}$, let
$\bar{J} = \{1, \ldots, r\} \setminus J$. There exist unique
polynomials $R_{\alphabar}$ in the Chern classes of the virtual
normal bundle $\nu(f)$, called {\em residue (or residual)
polynomials}, satisfying
\begin{equation}
\label{multiform} m_{\alphabar} = R_{\alphabar} + \sum_{1\in J
\subsetneq \{1, \ldots ,r\}} R_{\alphabar_{J}}
f^*(n_{\alphabar_{\overline{J}}})
\end{equation}
for admissible maps. Here the sum is taken over all possible subsets
of $\{1, \ldots, r \}$ containing 1. Moreover the residue
polynomials are independent of the order of the monosingularities
$\alpha_i$ in $\alphabar$.
\end{theorem}

In particular, if $\alphabar = (\alpha)$ is a monosingularity, then
(\ref{multiform}) yields $m_{\alpha} = R_{\alpha}$, hence
$R_{\alpha}$ is the Thom polynomial $\tp_\alpha$ of the given
singularity. For example, for $n=m$ we have $R_{A_2}=c_1^2+c_2$; and
this means that the cohomology class represented by points in $M$
where the map has singularity $A_2$, is equal to $c_1^2+c_2$ of the
virtual normal bundle of the map.

Observe that in the last example we did not specify $m$ and $n$,
only their difference. This is a classical fact about Thom
polynomials: the Thom polynomial of singularities having the same
local algebra and the same relative dimension $\ell$ (but maybe
living in different vector spaces $\E^0(m,n)$) are the same.

\smallskip

We can set $S_{\alphabar} = f_!(R_{\alphabar})$, and putting
(\ref{multiform}) together with (\ref{nalpha}) and the adjunction
formula for the Gysin map yields
\begin{equation}
\label{targetform} n_{\alphabar} = S_{\alphabar} + \sum_{1\in J
\subsetneq \{1, \ldots ,r\}} S_{\alphabar_{J}}
n_{\alphabar_{\overline{J}}}.
\end{equation}

\section{Calculation of $R_{\alphabar}$}
\label{clalcofR}

Different calculational techniques for residue polynomials of
monosingularities, that is, Thom polynomials of contact
singularities has been studied for decades. One of the most
effective techniques, which also generalizes to residue polynomials
of multisingularities was invented by the second author. We will call it
the interpolation method, and summarize it below. For more details
and proofs see \cite{rrtp}, \cite{lfrr04}

\subsection{Interpolation}
\label{interpolation}

Let $\xi$ be a contact singularity, and let us choose a {\em stable}
representative $\xi'\in \E^0(m,n)$. Stability of germs is discussed
e.g. in \cite{agvl}. In later sections we will not distinguish $\xi$
from $\xi'$, and call both by the same name $\xi$.

One of the main ideas of \cite{rrtp} is---roughly speaking---that we
can pretend that $\xi'$ is a map. Then stability of the germ implies
that as a map it is admissible, hence formulas (\ref{multiform}) and
(\ref{targetform}) hold for it. However, the source and target
spaces of $\xi'$ are (germs of) vector spaces, their cohomology ring
is trivial, so the formulas are meaningless. The idea is that we
consider formulas (\ref{multiform}) and (\ref{targetform}) for
$\xi'$ in equivariant cohomology. For this we need a group action.

A pair $(\phi,\psi)$ is a symmetry of the germ $\xi'$, if $\phi$
(resp. $\psi$) is an invertible element in $\E^0(m,m)$ (resp.
$\E^0(n,n)$), and $$\xi'=\psi\circ\xi'\circ \phi^{-1}.$$ If $G$ is a
group of symmetries of $\xi'$, then all ingredients of formulas
(\ref{multiform}) and (\ref{targetform}) make sense in the
$G$-equivariant cohomology ring of $\C^m$ (resp. $\C^n$). The
equivariant cohomology of a vector space is the same as the
equivariant cohomology of the one point space, i.e. the ring of the
$G$-characteristic classes $H^*BG$. Observe also, that the map
${\xi'}^*: H^*_G(\C^n) \to H^*_G(\C^m)$ is the identity map of
$H^*BG$.

The fact that formulas (\ref{multiform}) and (\ref{targetform}) hold
for stable singularities in equivariant cohomology put strong
constraints on the residue polynomials.

\begin{example}\label{whitney}
\rm Consider the stable germ $\xi:(x,y)\mapsto (x^2,xy,y)$, called
Whitney umbrella. The group $G=U(1)\times U(1)$ is a group of
symmetries of $\xi$ with the representations
$$(\alpha,\beta)\cdot (x,y)=(\alpha x,\beta\bar{\alpha} y),\qquad
(\alpha,\beta)\cdot (u,v,w)=(\alpha^2 u, \beta v, \beta\bar{\alpha}
w), \qquad\qquad \big((\alpha,\beta)\in U(1)\times U(1)\big)$$ on
the source and target spaces respectively. Indeed,
$$\xi(\ (\alpha,\beta)\cdot (x,y)\ ) = (\alpha,\beta)\cdot
\xi(x,y).$$ We also use $\alpha$ and $\beta$ for the first Chern
classes of the two factors of $U(1)\times U(1)$. Then
$H^*BG=\Q[\alpha,\beta]$, and we have that
$$c(\xi)=\frac{(1+2\alpha)(1+\beta)(1+\beta-\alpha)}{(1+\alpha)(1+\beta-\alpha)}=1+(\beta+\alpha)+(\alpha\beta-\alpha^2)+(-\alpha^2\beta+\alpha^3)+\ldots.$$
The closure of the double point locus (in the source space) of the
map $\xi$ is $\{y=0\}$. Its cohomology class is therefore
$\beta-\alpha$, the equivariant Euler class of its normal bundle.
The cohomology class represented by the image of this map can be
calculated to be $2\beta$ (see lemma \ref{adjlemma} below). The
pullback map $\xi^*$ is an isomorphism (as for all germs), hence the
pullback of the cohomology class of the image of $\xi$ is $2\beta$.
One of Kazarian's formulas (\ref{multiform}) (for maps from 2
dimensions to 3 dimensions) states that the difference of these two
multisingularity classes is $R_{A_0^2}$. Hence we get that
$$(\beta-\alpha)-2\beta=R_{A_0^2}(c_1=\beta+\alpha, c_2=\alpha\beta-\alpha^2, \ldots)\in \Q[a,b].$$
This has only one solution for $R_{A_0^2}$, namely $R_{A_0^2}=-c_1$.
\end{example}

Conditions obtained from stable singularities often determine
uniquely the residue polynomials, as follows. Let $\alpha$ be a
multisingularity of codimension $d$, and suppose that there are only
finitely many monosingularities $\xi$ with codimension $\leq d$. For
each $\xi$ we can consider the maximal compact symmetry group
$G_\xi$ (see \cite{rl}) of a stable representative. It is explained
in \cite{lfrr07} that $G_\xi$ acts on the normal bundle of
$\xi\subset \E^0(m,n)$. Then, we have the following theorem.

\begin{theorem}\cite{lfrr07}\label{equations_are_enough}
Suppose the $G_\xi$-equivariant Euler class of the normal bundle of
the embedding $\xi\subset \E^0(m,n)$ is not a 0-divisor for all the
finitely many singularities $\xi$ with $\codim \xi\leq d$. If
formula (\ref{multiform}) holds for stable representatives of all
the finitely many $\xi$ with codimension $\leq d$ (in $G_\xi$
equivariant cohomology), then formula  (\ref{multiform}) holds for
all admissible maps.
\end{theorem}

Strictly speaking this theorem is proved in \cite{lfrr07} only for
monosingularities (since that was the object of the paper). However,
what is proved there, is that the map
$$\Q[c_1,c_2,\ldots]\to \oplus H^*(BG_\xi),$$
whose component functions are the evaluations of Chern classes at
the stable representatives of the $\xi$'s with codim $\leq d$, is
injective in degrees $\leq d$. This implies the result for
multisingularities as well.

Mather \cite{mather6} determined the codimensions in which moduli of
singularities occur: for large $\ell$ moduli occurs in codimension
$6\ell+9$. Calculations show that the condition in the theorem on
the Euler classes of the monosingularities of codimension $\leq
6\ell+8$ also hold.

\subsection{A sample Thom polynomial calculation.}
\label{TpA1} We will show how Theorem \ref{equations_are_enough} can
be used to find the Thom polynomial of $A_1$ (a classical result,
due to Giambelli, Whitney, Thom in various disguises). We will carry
out the calculation for general $\ell$.

The codimension of the $A_1$ singularity is $\ell+1$, hence
$\tp_{A_1}$ is a degree $\ell+1$ polynomial, such that
\begin{equation}
\label{Tpunknown} [\overline{A_1(f)}]= \tp_{A_1}(c(f))
\end{equation}
for any admissible map $f$. There are only two singularities with
codimension $\leq \ell+1$, namely: $A_0$ and $A_1$. Hence from
Theorem~\ref{equations_are_enough} we can deduce two constraints on
the $\tp_{A_1}$. It turns out that the constraint coming from $A_0$
is redundant, hence we will now consider the constraint coming from
$A_1$ itself. For this we need to choose a stable representative of
the singularity $A_1$. The general procedure of finding a stable
representative of a singularity given by its local algebra is called
``universal unfolding''. For $A_1$ we obtain the following germ
$\C^{\ell+1},0 \to \C^{2\ell+1},0$:
$$f: (x,y_1,\ldots,y_\ell)\mapsto
(x^2,xy_1,\ldots,xy_{\ell},y_1,\ldots,y_\ell).$$ The general
procedure to find the maximal compact symmetry group is described in
 \cite{rl}. For our germ we obtain $G_f=U(1)\times
U(\ell)$ with the representations
$$\rho_1 \oplus (\overline{\rho_1} \otimes \rho_\ell),
\qquad \rho_1^2\oplus \rho_\ell \oplus(\overline{\rho_1} \otimes
\rho_\ell)$$ on the source and target spaces, where $\rho_1$ and
$\rho_\ell$ are the standard representations of $U(1)$ and
$U(\ell)$. It is easier to understand the representations of the maximal
torus $U(1)\times U(1)^\ell$, so we proceed as follows. Let
$(\alpha,\beta_1,\ldots,\beta_\ell)\in U(1)\times U(1)^\ell$. The
diagonal actions given by
$$(\alpha,\bar{\alpha}\beta_1,\ldots,\bar{\alpha}\beta_\ell),\qquad\text{and}\qquad
(\alpha^2,\beta_1,\ldots,\beta_\ell,\bar{\alpha}\beta_1,\ldots,\bar{\alpha}\beta_\ell)$$
is clearly a symmetry of the germ above.

Hence, when we apply formula (\ref{Tpunknown}) to the germ $f$, we
obtain an equation in $H^*(B(U(1)\times U(\ell)))$. By abuse of
language we denote the Chern roots of $U(1)$ and $U(\ell)$ by
$\alpha$ and $\beta_1,\ldots,\beta_\ell$. Let $b_i$ be the $i$'th
elementary symmetric polynomial of the $\beta_i$'s, that is the
universal Chern classes of the group $U(\ell)$. Then the total Chern
class of $f$ is
$$c(f)=\frac{(1+2\alpha)\prod^\ell (1+\beta_i) \prod^\ell
(1+\beta_i-\alpha)}{(1+\alpha)\prod^\ell (1+\beta_i-\alpha)}=
\frac{(1+2\alpha)\prod^\ell (1+\beta_i) }{(1+\alpha)}=$$
$$=1+(b_1+\alpha)+(b_2+b_1\alpha-\alpha^2)+(b_3+b_2\alpha-b_1\alpha^2+\alpha^3)+\ldots,$$
that is, $c_1(f)=b_1+\alpha$, $c_2(f)=b_2+b_1\alpha-\alpha^2$, etc.

Now we need the left hand side of formula (\ref{Tpunknown}) for our
germ $f$. The $A_1$ locus of the germ $f$ is only the origin, hence
$[A_1(f)]$ is the class represented by the origin. By definition the
class represented by the origin in the equivariant cohomology of a
vector space is the Euler class (a.k.a. top Chern class) of the
representation. In our case it is
$$\alpha\prod^{\ell}(\beta_i-\alpha).$$
Hence formula (\ref{Tpunknown}) reduces to
$$\alpha\prod^{\ell}(\beta_i-\alpha)=\tp_{A_1}(c_1=b_1+\alpha,
c_2=b_2+b_1\alpha-\alpha^2, \ldots).$$ It is simple algebra to show
that the polynomials $b_1+\alpha, b_2+b_1\alpha-\alpha^2, \ldots$
(up to the degree $\ell+1$ one) are algebraically independent in
$\Q[\alpha,b_1,b_2,\ldots,b_{\ell}]$, and that
$c_{\ell+1}=\alpha\prod^{\ell}(\beta_i-\alpha)$. This yields that
$\tp_{A_1}=c_{\ell+1}.$

\begin{remark} \label{genotype}
The ingredients of Kazarian's formulas (\ref{multiform}) are certain
geometrically defined classes ($m_\alpha$, $n_\alpha$), as well as
the Chern classes of the map. When applying these formulas for
stable representatives of monosingularities, there is an essential
simplification concerning only the Chern classes. The stable
representatives are universal unfoldings of so-called {\em
genotypes} of the singularity. The fact is that the genotype has the
same symmetry group as its universal unfolding; moreover, the Chern
classes of the genotype are the same as the Chern classes of its
universal unfolding, see \cite{rrtp}. Hence, later in the paper, if
we only need the Chern classes of a stable representative, we may
work with the genotype instead.
\end{remark}

\begin{remark} \label{torus} The ring of characteristic classes of a
group $G$ embeds into the ring of characteristic classes of its
maximal torus $T$. Hence the information that a formula holds in
$H^*(BG)$ is the same as that it holds in $H^*(BT)$. In what follows
we will always use the one more convenient for our notation.
\end{remark}

\section{The known general residue polynomials} \label{known}

Infinitely many Thom polynomials can be named at the same time, due
to certain stabilization properties that they satisfy. In the present
paper we will be concerned with two of the stabilizations. The first
we already mentioned, namely that the Thom polynomial only depends
on $\ell$, not on $m$ and $n$ (for the same local algebra). The
second---Theorem~\ref{thomseries} below---concerns the Thom polynomial as
$\ell$ varies (while not changing the local algebra). To phrase Theorem
\ref{thomseries} we need some notions.

Let $Q$ be a local algebra of a singularity. In singularity theory
one considers three integer invariants of $Q$ as follows: (i)
$\delta=\delta(Q)$ is the complex dimension of $Q$, (ii) the {\em
defect} $d=d(Q)$ of $Q$ is defined to be the minimal value of $b-a$
if $Q$ can be presented with $a$ generators and $b$ relations; (iii)
the definition of the third invariant $\gamma(Q)$ is more subtle,
see \cite[\S6]{mather6}. The existence of a {\em stable} singularity
$(\C^m,0)\to (\C^n,0)$ with local algebra $Q$ is equivalent to the
conditions $\ell\geq d$, $\ell(\delta-1)+\gamma\leq m$. Under these
conditions the codimension of the contact singularity with local
algebra $Q$ in $\E^0(m,n)$ is $\ell(\delta-1)+\gamma$.

\begin{theorem} \cite{lfrr07} \label{thomseries} Let $Q$ be a local algebra of singularities.
Assume that the normal Euler classes of the singularities in
$\E^0(m,n)$ with local algebra $Q$ are not 0. Then associated with
$Q$ there is a formal power series ({\em Thom series}) $\ts_Q$ in
the variables $\{d_i|i\in \Z\}$, of degree $\gamma(Q)-\delta(Q)+1$,
such that all of its terms have $\delta(Q)-1$ factors, and the Thom
polynomial of $\eta\subset \E^0(m,n)$ with local algebra $Q$ is
obtained by the substitution $d_i=c_{i+(m-n+1)}$. \qed
\end{theorem}

Even though there are powerful methods by now to compute individual
Thom polynomials (i.e.~finite initial sums of the $\ts$), finding
closed formulas for these Thom series remains a subtle problem. Here
are some examples.

\begin{description}
\item[\bf A$_0$] $Q=\C$ ({\sl embedding}). Here $\delta=1$, $\gamma=0$, and
\[\ts=1.\]
\item[\bf A$_1$] $Q=\C[x]/(x^2)$ (e.g. {\sl fold, Whitney umbrella}). Here $\delta=2$, $\gamma=$1, and
\[\ts=d_0.\]
\item[\bf A$_2$] $Q=\C[x]/(x^3)$ (e.g. {\sl cusp}). Here $\delta=3$, $\gamma=2$, and (see
\cite{rongaij})
\[\ts=d_0^2+d_{-1}d_1+2d_{-2}d_2+4d_{-3}d_3+8d_{-4}d_4+\ldots. \]
\item[\bf A$_3$] $Q=\C[x]/(x^4)$. Here $\delta=4$, $\gamma=3$, and (see \cite[Thm.4.2]{a3},
\cite{bsz06})
\[\ts=\sum_{i=0}^\infty 2^i
d_{-i}d_0d_i+\frac{1}{3}\sum_{i=1}^\infty \sum_{j=1}^\infty 2^i3^j
d_{-i}d_{-j}d_{i+j}+\frac{1}{2} \sum_{i=0}^\infty \sum_{j=0}^\infty
a_{i,j}d_{-i-j}d_id_j,\] where $a_{i,j}$ is defined by the formal
power series
\[\sum_{i,j}
a_{i,j}u^iv^j=\frac{u\frac{1-u}{1-3u}+v\frac{1-v}{1-3v}}{1-u-v}.\]
\end{description}

Although we used formal power series to describe Thom polynomials,
of course, the Thom polynomials themselves are polynomials, since
only finitely many terms survive for any concrete $\ell$. For
example from the Thom series of $A_2$ above it follows that for
$\ell=1$ the Thom polynomial is $c_1^2+c_2$, for $\ell=2$ the Thom
polynomial is $c_2^2+c_1c_3+2c_4$, etc.

\smallskip

There are other Thom series known in {\em iterated residue form}:
B\'erczi and Szenes found the  Thom series of $A_i$ singularities
for $i\leq 6$ \cite{bsz06}. In an upcoming paper \cite{lfrrpp} the
Thom series corresponding to several non-Morin singularities are
calculated. In \cite{balazs-laci} the Thom series of some second
order Thom-Boardman singularities are calculated.

\smallskip

However, all the mentioned results are Thom polynomials, that is
residue polynomials of monosingularities, rather than
multisingularities. Several individual multisingularity residue
polynomials are calculated for small $\ell$ in \cite{kazamulti} and
\cite{kaza}. However, the methods used there do not easily extend to
find formulas for all $\ell$. For example it was known that
\begin{eqnarray}
R_{A_0^4}= &  -6(c_1^3 + 3c_1c_2+2c_3) & \qquad \text{for } \ell=1, \\
R_{A_0^4}= & -6({c_{{2}}}^{3}+3\,c_{{1}}c_{{2}}c_{{3}}+7\,c_{{2}}c_{{4}}+2\,{c_{{1}}}^
{2}c_{{4}}+10\,c_{{1}}c_{{5}}+12\,c_{{6}}+{c_{{3}}}^{2}) &  \qquad \text{for } \ell=2,
\end{eqnarray}
but no $R_{A_0^4}$ formula was known for all $\ell$. In other words
the {\em residue series}, i.e. a formula containing $\ell$ as a
parameter is known only for a very few multisingularities. Here is a
complete list of those:

\begin{theorem} \label{a2th}\cite{rongamulti}
For admissible maps $f:M^m\to N^n$ we have
$$m_{A_0^2}=f^*(n_{A_0})-c_{\ell}(f).$$
That is, the residue polynomial of the multisingularity $A_0^2$ is
$-c_{\ell}$.
\end{theorem}

\begin{theorem} \label{a3th} \cite{ddlb96}
For admissible maps $f:M^m\to N^n$ we have
$$m_{A_0^3}= f^*(n_{A_0^2}) - 2c_\ell f^*(n_{A_0})
+ 2 \Big(c_\ell^2 + \sum_{i=0}^{\ell-1}2^i
c_{\ell-1-i}c_{\ell+1+i}\Big).
$$ That is, the residue polynomial of the multisingularity $A_0^3$
is
$$ R_{A_0^3} = 2 \Big(c_\ell^2 + \sum_{i=0}^{\ell-1} 2^i c_{\ell-1-i}c_{\ell+1+i}\Big). $$
\end{theorem}

\begin{theorem} \cite{kazamorin}
For admissible maps $f:M^m\to N^n$ we have
$$m_{A_1A_0}=f^*(n_{A_0}) - 2 \Big(c_\ell c_{\ell+1} + \sum_{i=0}^{\ell-1} 2^i c_{\ell-1-i}c_{\ell+2+i} \Big),$$
\begin{eqnarray*}
m_{A_0A_1} &= & f^*(n_{A_1}) - 2 \Big(c_\ell c_{\ell+1} + \sum_{i=0}^{\ell-1} 2^i c_{\ell-1-i}c_{\ell+2+i} \Big) \\
           &= & f^*(f_!(c_{\ell+1})) - 2 \Big(c_\ell c_{\ell+1} + \sum_{i=0}^{\ell-1} 2^i c_{\ell-1-i}c_{\ell+2+i}
           \Big).
\end{eqnarray*}
That is, the residue polynomial of the multisingularity $A_1A_0$ is
$$ R_{A_0A_1} = -2 \Big(c_\ell c_{\ell+1} + \sum_{i=0}^{\ell-1} 2^i c_{\ell-1-i}c_{\ell+2+i} \Big).$$
\end{theorem}

There are basically two main reasons why the calculation of other
residue polynomials is more difficult.

First, no transparent geometric meaning of residue polynomials of
multisingularities has been found so far. While residue polynomials
of monosingularities are equivariant classes represented by
geometrically relevant varieties in $\E^0(m,n)$, hence they are part
of equivariant cohomology, the residue polynomials of
multisingularities do not seem to be part of equivariant cohomology.
In other words, the cohomology ring of the classifying space of
singularities is a ring of characteristic classes, while the
cohomology ring of the classifying space of multisingularities
contains Landweber-Novikov classes (see more details in
\cite{kazamulti}). Hence powerful techniques of equivariant
cohomology (e.g.~localization) can not be used directly for
multisingularities.

The second reason can be seen in the diagram of multisingularities
in Section \ref{multi}. The codimension of the multisingularities
considered in the above three theorems are smaller than the
codimension of any non-Morin, (i.e. $\Sigma^{\geq 2}$) singularity.
Therefore, non-Morin singularities can be disregarded when studying
those three multisingularities. As the table shows, we will have
``competing'' non-Morin singularities for any other
multisingularity.

The main result of the present paper is the calculation of residue
polynomials in such non-Morin cases, namely the residue polynomial
$R_{A_0^i}$ for all $\ell$ and $i\leq 7$.

\section{General quadruple point formula}
\label{4-tuple}
 In order to emphasize the relative dimension,
let $R_{\alphabar}(\ell)$ denote the residue polynomial of the
multisingularity $\alphabar$ for maps of relative dimension $\ell$.
We are now ready to state the main theorem.

\begin{theorem} \label{main} For $i \leq 6$ we have
$$R_{A_0^{i+1}}(\ell) = (-1)^{i}i!R_{A_{i}}(\ell-1).$$
\end{theorem}

Since the polynomial $R_{A_i}$ is known for $i\leq 6$ \cite{bsz06}
this theorem calculates the polynomial $R_{A_0^i}$, hence determines
e.g. the general quadruple point formula. After some preparations,
the proof for the case $i=3$ will be given in Section \ref{proof}.
The cases $i=4,5,6$ follow similarly, see Section~\ref{moduli}.

\subsection{Multiple point formulas for germs.}

In what follows let us set the cohomology classes in the source and
target of the set of $j$-tuples of points of a map $f$ as
$\bar{m}_j(f)$, and $\bar{n}_j(f)$ respectively. That is,
$\bar{m}_j(f) = \bar{m}_{A_0^j}(f)$ and $\bar{n}_j(f) =
\bar{n}_{A_0^j}(f)$. We also use $n_1$ for $\bar{n_1}$. Using these
notations the defining equations (\ref{multiform}) of $R_{A_0^i}$'s
can be brought to the following form
\begin{eqnarray} \label{mieq}
\label{m2eq}\bar{m}_2(f) &=& f^*(\bar{n}_1(f)) + R_{A_0^2}(\ell) \\
\label{m3eq}\bar{m}_3(f) &=&f^*(\bar{n}_2(f)) + R_{A_0^2}(\ell)f^*(\bar{n}_1(f)) + \frac{1}{2}R_{A_0^3}(\ell)\\
\label{m4eq} \bar{m}_4(f)& =& f^{*}(\bar{n}_3(f)) +
R_{A_0^2}(\ell)f^{*}(\bar{n}_2(f)) +
\frac{1}{2}R_{A_0^3}(\ell)f^{*}(\bar{n}_1(f)) +
\frac{1}{6}R_{A_0^4}(\ell).
\end{eqnarray}

We want to apply the method of interpolation from Section
\ref{interpolation}, hence we want to apply equations
(\ref{m2eq})-(\ref{m4eq}) for stable germs with relative dimension
$\ell$, whose codimensions do not exceed the codimension of the
relevant $\bar{m}_i$. For stable germs, however, more information is
available for some of the ingredients.

\begin{lemma} \label{adjlemma}
Let $f$ be a stable germ with relative dimension $\ell$; and let $G$
be a symmetry group of $f$ with representations $\rho_0$ and
$\rho_1$ on the source and target spaces respectively. For a
$G$-representation $\rho$ let $e(\rho)$ denote the $G$-equivariant
Euler class of $\rho$, that is, the product of the weights of
$\rho$. Then in $G$-equivariant cohomology we have
\begin{itemize}
\item{} $f^*$ is isomorphism;
\item{} $f^*(n_1) e(\rho_0)=f^*(e(\rho_1))$;
\item{} $f^*(\bar{n_r})=\frac{1}{r} \bar{m}_r f^*(n_1)$.
\end{itemize}
\end{lemma}

\begin{proof} The map $f$ is equivariantly homotopic to the map of a
one point space to a one point space, hence $f^*:H^*BG\to H^*BG$ is
the identity map.

Now recall the adjunction formula for the Gysin map $f_!$ (which
holds for any proper map, therefore for any stable map germ too):
$$f_!(f^*(x)y)=xf_!(y).$$
Applying $f^*$ to this formula, and writing $z$ for $f^*(x)$, and
substituting $y=1$ we obtain
\begin{equation}
\label{germadjunction} f^*(f_!(z))=z f^*(n_1),
\end{equation}
where we also used that $f_!(1)$ is $n_1$. Since $f^*$ is an
isomorphism (hence surjective) this formula holds for any $z$.

Observe that $f_!(e(\rho_0))=e(\rho_1)$. Indeed, the Poincar\'e dual
of $e(\rho_0)$ is the homology class of 0 in the source, its
homology push-forward is the homology class 0 in the target, whose
Poincar\'e dual is then $e(\rho_1)$. Therefore substituting
$z=e(\rho_0)$ in (\ref{germadjunction}) we obtain the second
statement of the lemma.

Observe that $f_!(\bar{m_r})=r \bar{n_r}$. Therefore substituting
$z=\bar{m_r}$ into (\ref{germadjunction}) we obtain the third
statement.
\end{proof}

\begin{remark} Since $f^*$ is an isomorphism for germs, we will
sometimes suppress it from the notation. Observe that if
$e(\rho_0)\not=0$ then the second statement can be rewritten as
$f^*(n_1)=e(\nu(f))$, the equivariant Euler class of the virtual
normal bundle. The divisibility of $e(\rho_1)$ with $e(\rho_0)$ is a
remarkable property of stable germs. For instance it does not hold
for the non-proper blow-up map $(x,y)\to (x,xy)$ with group
$U(1)\times U(1)$ acting via $\rho_0=\alpha\oplus \beta$,
$\rho_1=\alpha\oplus (\alpha\otimes \beta)$.
\end{remark}

Using the statements of Lemma~\ref{adjlemma} we can bring formulas
(\ref{m2eq})-(\ref{m4eq}) to the forms

\begin{eqnarray}
\label{nm2eq}\bar{m}_2(f) &=& R_{A_0^2}(\ell) + n_1, \\
\label{nm3eq}\bar{m}_3(f) &=& \frac{1}{2}R_{A_0^3}(\ell)+ n_1 \big( \ldots\big), \\
\label{nm4eq}\bar{m}_4(f) & =& \frac{1}{6}R_{A_0^4}(\ell) + n_1
\big( \ldots \big),
\end{eqnarray}
where $n_1(\ldots)$ stands for a term divisible by $n_1$.

We will use these formulas to calculate certain substitutions of
residue polynomials. The variables of these polynomials are
$c_1,c_2,\ldots$. We will use the following notation for polynomials
$p$ with those variables: $p(1+x_1+x_2+\ldots)$ will denote the
substitution $c_1=x_1$, $c_2=x_2, \ldots$. Furthermore, the series
$1+x_1+x_2+\ldots$ will be usually given by (the Taylor series of) a
rational function. For example,
$$p\big( \frac{1+2\alpha }{1+\alpha }\big)$$
means the polynomial $p$ with substitution $c_1=\alpha$,
$c_2=-\alpha^2$, $c_3=\alpha^3$, etc.

\subsection{Some stable singularities and their
symmetries}\label{prototypes}

Along the way of proving Theorem~\ref{main} we will need the
following stable singularities.

$\bullet$ A stable $A_1$ singularity is $f_{A_1}=f_{A_1}(\ell):
\C^{\ell+1},0 \to \C^{2\ell+1},0$:
$$f_{A_1}: (x,y_1,\ldots,y_\ell)\mapsto
(x^2,xy_1,\ldots,xy_{\ell},y_1,\ldots,y_\ell).$$ Just like in
Section \ref{TpA1}, we consider its maximal compact symmetry group
$G=U(1)\times U(\ell)$ with the representations
$$\rho_1 \oplus (\overline{\rho_1} \otimes \rho_\ell),
\qquad \rho_1^2\oplus \rho_\ell \oplus(\overline{\rho_1} \otimes
\rho_\ell)$$ on the source and target spaces. For the $G$-equivariant
cohomology ring  we have,
$$H^*BG \leq \Q[\alpha,\beta_1,\ldots,\beta_\ell],$$
where $\alpha$, and $\beta_i$'s are the Chern roots of the
groups $U(1)$, and $U(\ell)$. Using this notation the total Chern
class of the virtual normal bundle of $f_{A_1}$ is
$$c(f_{A_1})=\frac{(1+2\alpha)\prod^{\ell}
(1+\beta_i)}{(1+\alpha)}.$$

Below is the list of the analogous data for singularities $A_2$,
$III_{2,2}$, and $A_3$.

$\bullet$ A stable $A_2$ singularity is $f_{A_2}=f_{A_2}(\ell):
\C^{2\ell+2},0 \to \C^{3\ell+2},0$:
$$f_{A_2}: (x,a,y_1,\ldots,y_\ell,z_1,\dots,z_\ell)\mapsto
(x^3+xa,x^2y_1+xz_1,\ldots,x^2y_{\ell}+xz_\ell,a,y_1,\ldots,y_\ell,z_1,\ldots
z_\ell).$$ Its maximal compact symmetry group $G=U(1)\times U(\ell)$
acts by the the representations
$$\rho_1 \oplus \rho_1^2 \oplus (\overline{\rho_1}^2 \otimes \rho_\ell) \oplus (\overline{\rho_1} \otimes \rho_\ell)
\qquad \rho_1^3 \oplus \rho_\ell \oplus \rho_1^2 \oplus (\overline{\rho_1}^2 \otimes \rho_\ell) \oplus (\overline{\rho_1} \otimes \rho_\ell)$$
on the source and target spaces. For the $G$-equivariant
cohomology ring we have,
$$H^*BG \leq \Q[\alpha,\beta_1,\ldots,\beta_\ell],$$
where $\alpha$, and $\beta_i$'s are the Chern roots of the
groups $U(1)$, and $U(\ell)$. Using this notation the total Chern
class of the virtual normal bundle of $f_{A_2}$ is
$$c(f_{A_2})=\frac{(1+3\alpha)\prod^{\ell}
(1+\beta_i)}{(1+\alpha)}.$$

$\bullet$ A stable $III_{2,2}$ singularity is $f_{III_{2,2}}=f_{III_{2,2}}(\ell):
\C^{2\ell+4},0 \to \C^{3\ell+4},0$:
$$f_{III_{2,2}}: (x_1,x_2,a,b,c,d,y_1,\ldots,y_{\ell-1},z_1,\ldots,z_{\ell-1}) \mapsto$$
$$(x_1x_2,x_1^2+cx_1+ax_2,x_2^2+bx_1+dx_2,y_1x_1+z_1x_2,\ldots, y_{\ell-1}x_1+z_{\ell-1}x_2,a,b,c,d,y_1,\ldots,y_{\ell-1},z_1,\ldots,z_{\ell-1})$$
We consider its symmetry group $G=U(1)\times U(1)\times U(\ell-1)$
with the representations
$$\rho_1 \oplus \rho_2 \oplus (\rho_1^2 \otimes \overline{\rho_2}) \oplus (\overline{\rho_1}\otimes \rho_2^2) \oplus \rho_1 \oplus \rho_2
\oplus (\rho_\ell \otimes \overline{\rho_1}) \oplus (\rho_\ell \otimes \overline{\rho_2}),
$$
$$ (\rho_1 \otimes \rho_2) \oplus \rho_1^2 \oplus \rho_2^2 \oplus \rho_\ell \oplus (\rho_1^2 \otimes \overline{\rho_2}) \oplus (\overline{\rho_1}\otimes \rho_2^2) \oplus \rho_1 \oplus \rho_2
\oplus (\rho_\ell \otimes \overline{\rho_1}) \oplus (\rho_\ell \otimes \overline{\rho_2})
$$

on the source and target spaces. For the $G$-equivariant
cohomology ring we have,
$$H^*BG \leq \Q[\alpha_1, \alpha_2,\beta_1,\ldots,\beta_{\ell-1}],$$
where $\alpha_i$, and $\beta_j$'s are the Chern roots of the respective
$U(1)$ groups, and of the group $U(\ell-1)$. Using this notation the total Chern
class of the virtual normal bundle of $f_{III_{2,2}}$ is
$$c(f_{III_{2,2}})=\frac{(1+2\alpha_1)(1+2\alpha_2)(1+(\alpha_1+\alpha_2))\prod^{\ell-1}
(1+\beta_i)}{(1+\alpha_1)(1+\alpha_2)}.$$

$\bullet$ A stable $A_3$ singularity is $f_{A_3}=f_{A_3}(\ell):
\C^{3\ell+3},0 \to \C^{4\ell+3},0$:
$$f_{A_3}: (x,a,b,w_1,\ldots, w_\ell,y_1,\ldots,y_\ell,z_1,\dots,z_\ell)\mapsto$$
$$(x^4+x^2a+xb,x^3w_1+x^2y_1+xz_1,\ldots,x^3w_\ell+x^2y_{\ell}+xz_\ell,a,b,w_1,\ldots,w_\ell,y_1,\ldots,y_\ell,z_1,\ldots z_l).$$
We consider its symmetry group $G_f=U(1)\times U(\ell)$
with the representations
$$\rho_1 \oplus \rho_1^2 \oplus \rho_1^3 \oplus (\overline{\rho_1}^3 \otimes \rho_\ell) \oplus (\overline{\rho_1^2} \otimes \rho_\ell) \oplus
(\overline{\rho_1} \otimes \rho_\ell)
$$
$$\rho_1^4 \oplus \rho_\ell \oplus \rho_1^2 \oplus \rho_1^3 \oplus (\overline{\rho_1}^3 \otimes \rho_\ell) \oplus (\overline{\rho_1^2} \otimes \rho_\ell) \oplus
(\overline{\rho_1} \otimes \rho_\ell)$$ on the source and target spaces. The $G$-equivariant
cohomology ring
$$H^*BG \leq \Q[\alpha,\beta_1,\ldots,\beta_\ell],$$
where $\alpha$, and $\beta_i$'s are the Chern roots of the
groups $U(1)$, and $U(\ell)$. Using this notation the total Chern
class of the virtual normal bundle of $f_{A_3}$ is
$$c(f_{A_3})=\frac{(1+4\alpha)\prod^{\ell}
(1+\beta_i)}{(1+\alpha)}.$$

\subsection{Proof of Theorem \ref{main}} \label{proof}

Now we prove Theorem \ref{main} for $i=3$, that is
\begin{equation}
R_{A_0^4}(\ell)=-6 R_{A_3}(\ell-1).
\end{equation}

\begin{proof}
Consider the stable singularity of type $A_1$ with relative
dimension $\ell$ from Section \ref{prototypes}. This map $f_{A_1}$
has no quadruple point, hence $\bar{m}_4=0$ for it. This can be checked
directly, or using the fact from singularity theory that the highest
multiple points of a stable singularity with local algebra $Q$ are
the $\delta$-tuple points, where $\delta$ is the dimension of the
local algebra $Q$.

Lemma \ref{adjlemma} above shows that for $f_{A_1}$ we have
$$n_1(f_{A_1})=\frac{2\alpha \prod^\ell\beta_i \prod^{\ell}(\beta_i-\alpha)}{\alpha
\prod^{\ell}(\beta_i-\alpha)}=2\prod^{\ell} \beta_i.$$ Thus, for
$f_{A_1}$ equation (\ref{nm4eq}) becomes
$$0=\frac{1}{6} R_{A_0^4}(\ell)\big( \frac{(1+2\alpha)\prod^\ell
(1+\beta_i)}{(1+\alpha)}\big) + \big(2\prod^\ell \beta_i \big) \big(\ldots
\big).$$ Plugging in $\beta_{\ell}=0$ we obtain
\begin{equation}
0=R_{A_0^4(\ell)}\big( \frac{(1+2\alpha)\prod^{\ell-1}
(1+\beta_i)}{(1+\alpha)}\big). \label{q1}
\end{equation}

We repeat the above arguments for the stable singularity of type $A_2$ with relative
dimension $\ell$, and we obtain

\begin{equation}
0=R_{A_0^4(\ell)}\big( \frac{(1+3\alpha)\prod^{\ell-1}
(1+\beta_i)}{(1+\alpha)}\big). \label{q2}
\end{equation}

The argument for the $III_{2,2}$ singularity is similar.
We have $$n_1(f_{III_{2,2}})=\frac{2\alpha_1 2\alpha_2 (\alpha_1+\alpha_2) \prod^{\ell-1}\beta_i \prod^{\ell-1}(\beta_i-2\alpha_1)(\beta_i-2\alpha_2)(\beta_i-\alpha_1-\alpha_2)}{\alpha_1 \alpha_2
\prod^{\ell-1}(\beta_i-2\alpha_1)(\beta_i-2\alpha_2)(\beta_i-\alpha_1-\alpha_2)}=4(\alpha_1+\alpha_2)\prod^{\ell-1} \beta_i.$$
Thus for $f_{III_{2,2}}$ equation (\ref{nm4eq}) becomes
$$0=\frac{1}{6}R_{A_0^4(\ell)}\big( \frac{(1+2\alpha_1)(1+2\alpha_2)(1+(\alpha_1+\alpha_2))\prod^{\ell-1}
(1+\beta_i)}{(1+\alpha_1)(1+\alpha_2)}\big) +
4(\alpha_1+\alpha_2)\prod^{\ell-1} \beta_i \big( \ldots \big).$$
Substituting $\beta_{\ell-1} = 0$ we obtain
\begin{equation}
0=R_{A_0^4(\ell)}\big( \frac{(1+2\alpha_1)(1+2\alpha_2)(1+(\alpha_1+\alpha_2))\prod^{\ell-2}
(1+\beta_i)}{(1+\alpha_1)(1+\alpha_2)}\big). \label{q3}
\end{equation}

Now consider the stable singularity of type $A_3$ with relative
dimension $\ell$ from Section \ref{prototypes}. The closure of the
quadruple point set of $f_{A_3}$ in the source space is $\{w_i=0,
y_i=0, z_i=0\}$. Hence for $f_{A_3}$ we have $\bar{m}_4=$ Euler
class of the normal bundle to $\{w_i=0, y_i=0, z_i=0\}$. That is
$$\bar{m}_4=\prod^{\ell}(\beta_i-\alpha)(\beta_i-2\alpha)(\beta_i-3\alpha).$$

\noindent Lemma \ref{adjlemma} above shows that for $f_{A_3}$ we have
$$n_1(f_{A_3})=4\prod^{\ell} \beta_i.$$ Thus, for
$f_{A_3}$ equation (\ref{nm4eq}) becomes

$$\prod^{\ell}
(\beta_i-\alpha)(\beta_i-2\alpha)(\beta_i-3\alpha) =\frac{1}{6}
R_{A_0^4}(\ell)\big( \frac{(1+4\alpha)\prod^\ell
(1+\beta_i)}{(1+\alpha)}\big) + 4\prod^\ell \beta_i \big(\ldots
\big).$$ Plugging in $\beta_{\ell}=0$ we obtain
\begin{equation}
-6\alpha^3 \prod^{\ell-1}
(\beta_i-\alpha)(\beta_i-2\alpha)(\beta_i-3\alpha)=\frac{1}{6}R_{A_0^4(\ell)}\big(
\frac{(1+4\alpha)\prod^{\ell-1}
(1+\beta_i)}{(1+\alpha)}\big).\label{q4}
\end{equation}

Observe that formulas (\ref{q1}), (\ref{q2}) (\ref{q3}) and
(\ref{q4}) mean that the polynomial $-\frac{1}{6}R_{A_0^4}(\ell)$
satisfies the following properties: (i) it vanishes when applied to
$f_{A_1}(\ell-1)$, $f_{A_2}(\ell-1)$, $f_{III_{2,2}}(\ell-1)$ ; (ii)
it gives the Euler class of the source space when applied to
$f_{A_3}(\ell-1)$. These are exactly the properties of the
polynomial $R_{A_3}(\ell-1)$ applied to these four singularities.
According to Theorem~\ref{equations_are_enough}, these properties
determine $R_{A_3}(\ell-1)$, hence we have proven that
$R_{A_0^4}(\ell)=-6R_{A_3}(\ell-1)$.
\end{proof}

In summary we obtained the {\em general quadruple point formula}:
\begin{equation}\label{general-main}
m_4 = f^*(n_3) - 3c_\ell f^*(n_2) + 6 \Big(c_\ell^2 +
\sum_{i=0}^{\ell-1} 2^i c_{\ell-1-i}c_{\ell+1+i}\Big)f^*(n_1) +
p(c_i)
\end{equation}
\noindent where
$$ p(c_i) = R_{A_0^4}= -6 \Bigg(\sum_{i=0}^\infty 2^i c_{\ell-i}c_\ell c_{\ell+i} + \frac{1}{3} \sum_{i=1}^\infty \sum_{j=1}^\infty 2^i3^jc_{\ell-i}c_{\ell-j}c_{\ell+i+j}
+ \frac{1}{2} \sum_{i=0}^\infty \sum_{j=0}^\infty
a_{i,j}c_{\ell-i-j}c_{\ell+i}c_{\ell+j} \Bigg)$$ \noindent where
$c_0=1, c_{<0}=0$ and with the $a_{i,j}$'s defined as for the Thom
polynomial of $A_3$, that is:
\[\sum_{i,j}
a_{i,j}u^iv^j=\frac{u\frac{1-u}{1-3u}+v\frac{1-v}{1-3v}}{1-u-v}.\]

\begin{remark}Another way of viewing the $a_{i,j}$'s is as the entries of the following,
modified Pascal's triangle
\begin{tabular}{ccccccccccccccccccc}
  &   &   &   & $a_{0,0}$ &   &   &   &  &   &   &   &  &   &  0 &   &  \\
  &   &   & $a_{1,0}$ &   & $a_{0,1}$ &   &   &   &   &   &  & &  1&   &  1 &\\
  &   & $a_{2,0}$ &   & $a_{1,1}$ &   & $a_{0,2}$ &   &  &  = &  &   & 3 &   &  2 &   & 3\\
  & $a_{3,0}$ &   & $a_{2,1}$ &   &  $a_{1,2}$ &   & $a_{0,3}$ &  &   &   &   9&   &  5  &   & 5 & &9 \\
$a_{4,0}$ &   & $a_{3,1}$ &   &  $a_{2,2}$ &   & $a_{1,3}$ &   & $a_{0,4}$& & 27 &   & 14 &   &  10 &   & 14 &   & 27 \\
\end{tabular}
\smallskip

\noindent where the rule for the $i,j$th entry remains the same, but
we have placed powers of 3 on the edges instead of 1's.
\end{remark}

\subsection{Higher multiple point formulas} \label{moduli}
The proof of Theorem \ref{main} for $i=4,5,6$ goes along the same
line as for $i=3$. One considers the finitely many monosingularities
whose codimension is less than $(i+1)\ell$, as well as the
monosingularity $A_i(\ell)$. Applying the defining relation of
$R_{A_0^{i+1}}$ for these monosingularities (in equivariant
cohomology) results in certain formulas for different
specializations of $R_{A_0^{i+1}}$. Plugging in 0 for the ``last
Chern root'', just like in (\ref{q1}), one obtains some shorter,
simpler formulas, which turn out to mean that the residue polynomial
$(-1)^i/i!\cdot R_{A_0^{i+1}}(\ell)$ satisfies the exact same
substitutions as $R_{A_{i}}(\ell-1)$. Using the statement that these
substitutions determine $R_{A_{i}}(\ell-1)$ (Theorem
\ref{equations_are_enough}) we conclude that
$R_{A_0^{i+1}}(\ell)=(-1)^i i!R_{A_{i}}(\ell-1)$.

One naturally conjectures that $R_{A_0^{i+1}}(\ell)=(-1)^i i!
R_{A_i}(\ell-1)$ holds for all $i$. We found reasons supporting this
conjecture, but no proof. The method this article uses certainly
does not work for $i>6$. The reason is a 1-dimensional family of
singularities that together form a codimension $6\ell+9$ variety in
$\E^0(n,n+\ell)$ (for large $\ell$) \cite{mather6}. Hence, beyond
codimension $6\ell+8$ we can not apply
Theorem~\ref{equations_are_enough}.

\section{4-secants to Smooth Projective Varieties}
\label{4-secants}

As an application of the quadruple point formula we find the number
of 4-secant planes to smooth projective varieties. The method can be
tailored to find the number of 4-secant (or 5-, 6-, or 7-secant)
linear spaces of other dimensions.

Let $i:V^a\subset \P^{4a+2}$ be a smooth projective variety, and let
$\G=\Gr_2\P^{4a+2}=\Gr_3\C^{4a+3}$ denote the Grassmannian of
projective 2-planes in $\P^{4a+2}$. Consider the following incidence
varieties:
\begin{eqnarray*}
\B &:= &\{(x,P) \in V \times \G\ |\ x\in P \}, \\
\F &:= &\{(x,P) \in \P^{4a+2} \times \G\ |\ x \in P \}.
\end{eqnarray*}

The two projections of $\textbf{F}$ to $\P^{4a+2}$ and $\G$ will be
denoted by $q$ and $\pi$. Both are fibrations with fibers
$\Gr_2\C^{4a+2}$ and $\P^2$, respectively. The restriction of $q$ to
the variety $V$ is the fibration $p:\textbf{B}\to V$. Hence we
obtain the following diagram.

\begin{equation} \label{4secant}
        \begin{CD}
          \textbf{B}^{9a} @>j>> \textbf{F}^{12a+2} @>\pi>> \G^{12a}\\
            @VVpV             @VVqV \\
         V^a @>i>>    \Pn^{4a+2},
         \end{CD}
\end{equation}
where upper indexes mean dimensions. Observe that the quadruple
points of the map $f=\pi \circ j$ correspond bijectively to planes
intersecting $V$ exactly four times, ie. 4-secant planes.

We will make the assumption that the map $f$ is admissible (cf.
Remark~\ref{unpleasant}). Hence, the number $N_a$ of 4-secant planes
to $V$ is calculated as
\begin{eqnarray}
N_a& = &  \frac{1}{4!}\int_{\G} n_{A_0^4(f)} \\
   & = & \frac{1}{4!}\int_{\G}
f_!(R_{A_0})^4 +
6f_!(R_{A_0^2})f_!(R_{A_0})^2+3f_!(R_{A_0^2})^2+4f_!(R_{A_0^3})f_!(R_{A_0})+f_!(R_{A_0^4}),
\label{int}
\end{eqnarray}
where the polynomials $R_{A_0^i}$ are evaluated at the Chern classes
of the virtual normal bundle $\nu_f$ of~$f$. In the rest of this
section we show how this integral can be calculated.

First observe that $\nu_f=\nu_j \oplus j^*\nu_\pi=p^*(\nu_i)\ominus
j^*(\kappa)$, where $\kappa$ is the fiberwise tangent bundle to the
fibration $\pi$. Let the Chern classes of $\kappa$ be $k_1$, $k_2$,
and let the Chern classes of $\nu_i$ be $n_1,\ldots, n_a$.

Since the the polynomials $R_{A_0^i}$ are explicitly known (see
Section~\ref{known}), the integrand in (\ref{int}) is an explicit
polynomial, whose terms are of the form $f_!(j^*(k)p^*(n))$, where
$k$ is a monomial in $k_1$, $k_2$, and $n$ is a monomial in
$n_1,\ldots,n_a$. This term is further equal to
$$\pi_! j_! (j^*(k)p^*(n))=\pi_!(k j_!p^*(n))=\pi_!(k
q^*(i_!(n))).$$

The cohomology classes $i_!(n)$ are the geometric invariants of the
variety $V^a\subset \P^{4a+2}$---we want to calculate the number of
4-secant planes in terms of these invariants. These classes can be
encoded by integers, as follows.

\begin{definition} Let $h$ be the class represented by a hyperplane in
$H^*(\P^{4a+2})$, hence $H^*(\P^{4a+2})=\Q[h]/(h^{4a+3})$. For a
multiindex $u=(u_1,u_2,\ldots,u_a)$ let $\chi_u$ be the coefficient
of the appropriate power of $h$ in $i_!(n_1^{u_1}n_2^{u_2}\cdots
n_a^{u_a})$. (For example $\chi_{(0,\ldots,0)}$ is the degree of the
embedding $V\subset \P^{4a+2}$.)
\end{definition}

Using this notation, we obtain that our integrand can be written as
a linear combination of terms of the form $\pi_!(k\cdot q^*(h^w))$
($w\in \N$), with coefficients depending on the invariants $\chi_u$.

Let $S$ and $Q$ be the universal sub and quotient bundles over $\G$.
The space $\F$ is the projectivization of the bundle $S$.
Corresponding to this fact, we have the tautological exact sequence
of bundles $0\to l \to \pi^*S \to \pi^*S/l\to 0$ over $F$. Moreover,
$\kappa$ being the fiberwise tangent bundle, we have
$\kappa=l^*\otimes \pi^*S/l$. Using the fact that $q^*(h)$ is the
first Chern class of $l$, we obtain that the integrand can further
be written as linear combination of terms of the form
$$\pi_!( c_1(l)^w c_I(\pi^*(S))).$$
Here $c_I$ is any Chern monomial, and $w$ is a non-negative integer.
This term is further equal to
$$c_I(S)\pi_!(c_1(l)^w).$$

The cohomology ring of $\G$, together with the $\pi_!$-image of
powers of $c_1(l)$ are well known, see for example
\cite{fulton:int}:
$$H^*(\G)=Q[c_i(S),c_i(Q)]/(c(S)c(Q)=1),$$
$$\pi_!(c_1(l)^w)=c_{w-2}(Q).$$

Hence our integrand is an explicit class in $H^*(\G)$. Integration
can be utilized in any computer algebra package. The results we
obtain this way are as follows.

\begin{theorem}\label{ns}
Let $V^a\subset \P^{4a+2}$ be a smooth variety such that the
associated map $f:\B\to \G$ defined in (\ref{4secant}) is
admissible. Let $\chi_u$ be the invariants of the embedding. Then
for the number $N_a$ of 4-secant planes to $V^a$ we have

\begin{eqnarray*}
4!N_1& =&{{ \chi}_{{0}}}^{4}+24{ \chi}_{{1}}{ \chi}_{{0}}-6{
\chi}_{{1} }{{ \chi}_{{0}}}^{2}-208{{ \chi}_{{0}}}^{2}+24{{
\chi}_{{0}}}^{ 3}+3{{ \chi}_{{1}}}^{2}+1008{ \chi}_{{0}} -174{
\chi}_{{1}},
\end{eqnarray*}

\begin{eqnarray*}
4!N_2 &= &-36{ \chi}_{{1,0}}{ \chi}_{{0,1}}+64{ \chi}_{{2,0}}{
\chi}_{{0 ,0}}-3156{ \chi}_{{1,0}}{ \chi}_{{0,0}}+{{
\chi}_{{0,0}}}^{4}+36
{ \chi}_{{1,0}}{{ \chi}_{{0,0}}}^{2}\\
& & -6{ \chi}_{{0,1}}{{ \chi}_ {{0,0}}}^{2}-126{{
\chi}_{{0,0}}}^{3}+12075{{ \chi}_{{0,0}}}^{2}
+286{ \chi}_{{0,0}}{ \chi}_{{0,1}}-1356{ \chi}_{{2,0}}\\
& &-1944{
 \chi}_{{0,1}}-200838{ \chi}_{{0,0}}+3{{ \chi}_{{0,1}}}^{2}+108
{{ \chi}_{{1,0}}}^{2}+42174{ \chi}_{{1,0}},
\end{eqnarray*}

\begin{eqnarray*}
4!N_3 & = & -1728{ \chi}_{{1,0,0}}{ \chi}_{{0,1,0}}+91200{
\chi}_{{0,0,0}}{
 \chi}_{{1,0,0}}-6{ \chi}_{{0,0,1}}{{ \chi}_{{0,0,0}}}^{2}+384{
 \chi}_{{1,1,0}}{ \chi}_{{0,0,0}}
\\ & & -48{ \chi}_{{0,1,0}}{ \chi}_{{0
,0,1}}-144{ \chi}_{{0,0,0}}{ \chi}_{{0,0,1}}-4352{ \chi}_{{2,0,0
}}{ \chi}_{{0,0,0}}-26004{ \chi}_{{1,1,0}}+ \\
& &3{{ \chi}_{{0,0,1}}} ^{2}-9523080{ \chi}_{{1,0,0}}+{{
\chi}_{{0,0,0}}}^{4}+42058080{
 \chi}_{{0,0,0}}+614880{ \chi}_{{0,1,0}}\\
& &-23934{ \chi}_{{0,0,1}} -3156{ \chi}_{{3,0,0}}-448320{{
\chi}_{{0,0,0}}}^{2}+437400{
 \chi}_{{2,0,0}}+3888{{ \chi}_{{1,0,0}}}^{2}\\
& &+192{{ \chi}_{{0,1,0 }}}^{2}+720{{ \chi}_{{0,0,0}}}^{3}-5120{
\chi}_{{0,0,0}}{ \chi}
_{{0,1,0}}+216{ \chi}_{{1,0,0}}{ \chi}_{{0,0,1}} \\
& &-216{ \chi}_{{1 ,0,0}}{{ \chi}_{{0,0,0}}}^{2}+48{
\chi}_{{0,1,0}}{{ \chi}_{{0,0,0 }}}^{2},
\end{eqnarray*}

\begin{eqnarray*}
4!N_4 & = & 1280{ \chi}_{{1,0,1,0}}{ \chi}_{{0,0,0,0}}-1320{
\chi}_{{0,0,0, 1}}{ \chi}_{{1,0,0,0}}+853550{ \chi}_{{0,1,0,0}}{
\chi}_{{0,0,0,0
}}\\
&  & +1320{ \chi}_{{1,0,0,0}}{{ \chi}_{{0,0,0,0}}}^{2}-33024{ \chi
}_{{1,1,0,0}}{ \chi}_{{0,0,0,0}}+60{ \chi}_{{0,0,1,0}}{{ \chi}_{{0
,0,0,0}}}^{2}\\
& & -72600{ \chi}_{{0,1,0,0}}{ \chi}_{{1,0,0,0}}-330{
 \chi}_{{0,1,0,0}}{{ \chi}_{{0,0,0,0}}}^{2}-3300{ \chi}_{{0,0,1,0
}}{ \chi}_{{0,1,0,0}}\\
& & -60{ \chi}_{{0,0,1,0}}{ \chi}_{{0,0,0,1}}+ 6466{
\chi}_{{0,0,0,1}}{ \chi}_{{0,0,0,0}}-4290{{ \chi}_{{0,0,0
,0}}}^{3} \\
& & -6721080{ \chi}_{{1,0,0,0}}{ \chi}_{{0,0,0,0}}-92436{
 \chi}_{{0,0,1,0}}{ \chi}_{{0,0,0,0}}+247040{ \chi}_{{2,0,0,0}}{
 \chi}_{{0,0,0,0}}\\
& & +512{ \chi}_{{0,2,0,0}}{ \chi}_{{0,0,0,0}}+330 {
\chi}_{{0,0,0,1}}{ \chi}_{{0,1,0,0}}-309768{ \chi}_{{0,0,0,1}
}\\
& & +2126696220{ \chi}_{{1,0,0,0}}+13200{ \chi}_{{0,0,1,0}}{ \chi}
_{{1,0,0,0}}+300{{ \chi}_{{0,0,1,0}}}^{2}\\
& & +1272924{ \chi}_{{3,0,0 ,0}}+9382770{
\chi}_{{0,0,1,0}}-9023984640{ \chi}_{{0,0,0,0}}+
10379016{ \chi}_{{1,1,0,0}}\\
& & +3{{ \chi}_{{0,0,0,1}}}^{2}-104832{
 \chi}_{{0,2,0,0}}+145200{{ \chi}_{{1,0,0,0}}}^{2}-158489298{
 \chi}_{{0,1,0,0}}\\
& & -292860{ \chi}_{{1,0,1,0}}+9075{{ \chi}_{{0,1
,0,0}}}^{2}+{{ \chi}_{{0,0,0,0}}}^{4}-81576{ \chi}_{{2,1,0,0}}\\
& & -113973552{ \chi}_{{2,0,0,0}}+24962795{{
\chi}_{{0,0,0,0}}}^{2}-6 { \chi}_{{0,0,0,1}}{{
\chi}_{{0,0,0,0}}}^{2}.
\end{eqnarray*}

\end{theorem}

We note that the expression for $N_2$ has appeared in \cite{tt} and
\cite{lehn99}, in the language of Hilbert schemes (and in the
variables $d=\chi_{0,0}$, $\pi=\chi_{1,0}-11\chi_{0,0}$,
$\kappa=\chi_{2,0}-22\chi_{1,0}+121\chi_{0,0}$,
$e=-\chi_{0,1}+\chi_{2,0}-11\chi_{1,0}+55\chi_{0,0}$), but we
believe that $N_3$ and beyond are new results. Expressions for
$N_{>4}$, as well as formulas counting 4-secant linear spaces of
higher dimensions can be obtained similarly.

\begin{remark} \label{unpleasant} Theorem \ref{ns} contains the unpleasant
condition that the associated map is admissible. Looking through the
literature on enumerative geometry using topological methods we find
that authors explicitly or implicitly suppose similar admissibility
properties. Namely, the following seems to be a general belief: when
starting with a geometric situation one associates a map between
parameter spaces, and the map is not a Legendre or Lagrange map
(e.g. its relative dimension is $>1$), then the map is admissible,
provided some genericity condition holds. We are not able to phrase
(let alone prove) such a genericity condition, under which the
admissibility property of the associated map holds.
\end{remark}

\section{Another Multisingularity Formula}
\label{III22A0}

The interpolation method described in Section \ref{interpolation}
can be applied to find finite initial sums of the series describing
the general multisingularity polynomials. If the multisingularity is
complicated enough, recognizing and proving the pattern in such
final sums quickly become intractable. An exception is given by the
theorem below. We will use the following versions of Schur
polynomials
\begin{equation} \label{schur}
s(i,j,k) = \det \left( \begin{array}{ccc}
                   c_i & c_{i+1} & c_{i+2} \\
                   c_{j-1}& c_j& c_{j+1}\\
                   c_{k-2}&c_{k-1}& c_{k}
               \end{array} \right), \qquad
s(i,j) = \det \left( \begin{array}{ccc}
                   c_i & c_{i+1}  \\
                   c_{j-1}& c_j   \\
               \end{array} \right).
\end{equation}

\begin{theorem}\label{III22} The general $III_{2,2}A_0$-multisingularity
residue polynomial for maps of relative dimension~$\ell$ is
\begin{equation} \label{RHS} R_{III_{2,2}A_0}= -\sum_{i=1}^{\infty}
2^{i+1}s(\ell+1+i,\ell+2,\ell+1-i).
\end{equation}
\end{theorem}

\begin{proof}
Let us denote the right hand side of equation (\ref{RHS}) by
$\mathcal{R}$. We will show that $\mathcal{R}$ satisfies the
defining relation of the residue polynomial $R_{III_{2,2}A_0}$, that
is, we will show
\begin{equation}\label{III22main}
m_{III_{2,2}A_0} = \mathcal{R} + R_{III_{2,2}}n_{A_0}
\end{equation}
for all admissible maps. The Giambelli-Thom-Porteous formula states
that $R_{III_{2,2}}=s(\ell+2,\ell+2)$.
Theorem~\ref{equations_are_enough} asserts that if (\ref{III22main})
holds for stable representatives of $A_0$, $A_1$, $A_2$, $A_3$,
$I_{2,2}$, and $III_{2,2}$ singularities (in equivariant cohomology
with respect to the maximal compact symmetry group of the particular
singularity), then (\ref{III22main}) holds for admissible maps.
Below we prove these statements.

\subsection{Restriction to $A_r$ singularities.}
Stable representatives of $\ell$ relative dimensional $A_r$
singularities are universal unfoldings of germs $\C\to \C^{\ell+1}$
$$(x)\mapsto (x^{r+1},0,\ldots,0).$$
Their maximal compact symmetry group is $U(1) \times U(\ell)$. The
formal difference of the representation on the target and the source
is
$$\rho_1^{r+1} \oplus \rho_\ell - \rho_1,$$
where $\rho_1$ and $\rho_\ell$ are the standard representations of
the $U(1)$ and $U(\ell)$ factors. Therefore the Chern classes $c_i$
of the stable representative of $A_r$ are obtained by
\begin{equation}\label{aichern}
1+c_1t+c_2t^2+\ldots = \frac{1-(r+1)a t}{1-at}\sum_{i=0}^\ell
d_it^i,\qquad (d_0=1)\end{equation} where $-a$ is the first Chern
class of $U(1)$, and $d_i$ are the Chern classes of $U(\ell)$.
Observe that relation (\ref{aichern}) implies $c_{j+1}=a c_j$ for
$j\geq \ell+2$. Therefore the first two rows of each term of
$\mathcal{R}$ are linearly dependent, making each determinant 0.
Hence $\mathcal{R}=0$ applied to any $A_r$ singularity.

Since $III_{2,2}$ is a $\Sigma^2$ singularity, and all $A_r$'s are
$\Sigma^1$ singularities, near an $A_r$ singularity there are no
$III_{2,2}$ or $III_{2,2}A_0$ (multi)singularities. This implies
that $R_{III_{2,2}}=m_{III_{2,2}}$ and $m_{III_{2,2}A_0}$ applied to
stable representatives of all $A_r$ singularities are both 0.
Therefore we proved that (\ref{III22main}) holds for all $A_r$
singularities.

\subsection{Restriction to $I_{2,2}$ singularities.} Stable
singularities of type $I_{2,2}$ of relative dimension $\ell$ are
universal unfoldings of the germ $\C^2\to \C^{\ell+2}$
$$(x,y)\mapsto (x^2,y^2,0,\ldots,0).$$
The  maximal compact symmetry group of this germ is $U(1)^2 \times
U(\ell)$, and the formal difference of the representations of this
group on the target and on the source is:
$$\rho_1^2 \oplus \rho_1^{'2} \oplus \rho_{\ell} - (\rho_1 \oplus
\rho_1').$$ Here $\rho_1$ and $\rho_1'$ are the standard
representations of the two $U(1)$ factors, and $\rho_\ell$ is the
standard representation of $U(\ell)$. Therefore the Chern classes
$c_i$ of the stable representative of $I_{2,2}$ are obtained by
\begin{equation}\label{i22chern}
1+c_1t+c_2t^2+\ldots = \frac{(1-2a
t)(1-2bt)}{(1-at)(1-bt)}\sum_{i=0}^\ell d_it^i,\end{equation} where
$-a$ and $-b$ are the first Chern classes of the two $U(1)$ factors,
and $d_i$ are the Chern classes of $U(\ell)$.  We need the following
lemma.

\begin{lemma} Let $e_i$ and $h_i$ denote the elementary, and
complete symmetric polynomials of the variables $a$ and $b$ (e.g.
$e_2=ab$, $h_2=a^2+ab+b^2$). Suppose the variables $c_i$ are
expressed in terms of $a$, $b$, and $d_1,\ldots,d_\ell$ as defined
in (\ref{i22chern}). We use the convention that $d_0=1$, $d_{<0}=0$
and $d_{>\ell}=0$. Then
$$s(\alpha, \beta, \gamma)=e_2^{\beta-\ell-2}\cdot
h_{\alpha-\beta}\cdot s(\ell+2,\ell+2)\cdot
(d_\gamma-2e_1d_{\gamma-1}+4 e_2d_{\gamma-2}),$$ for $\alpha\geq
\beta\geq \gamma$, $\beta\geq \ell+2$, and $\gamma\leq \ell$.
\end{lemma}

\begin{proof} The Factorization Formula for Schur polynomials
(e.g. \cite{fp}) claims that substituting
$$1+c_1t+c_2t^2+\ldots=\frac{\sum_{i=0}^{\ell+2}
D_it^i}{(1-at)(1-bt)}$$ into $s(\alpha,\beta,\gamma)$ yields $
e_2^{\beta-\ell-2}h_{\alpha-\beta}s(\ell+2,\ell+2)D_\gamma$.
Carrying out the further substitution $\sum_{i=0}^{\ell+2}
D_it^i=(\sum_{i=0}^\ell d_it^i)(1-2at)(1-2bt)$ gives the statement
of the lemma.
\end{proof}

A special case of this lemma claims that for $j\geq 1$ we have
$$s(\ell+1+j,\ell+2,\ell+1-j)=h_{j-1}\cdot s(\ell+2,\ell+2) \cdot
(d_{\ell+1-j}-2e_1d_{\ell-j}+4e_2d_{\ell-1-j}).$$ Plugging this into
the formula for $\mathcal{R}$ we obtain a linear function of the
$d_i$ variables. The coefficient of $d_{\ell-k}$ for $k>0$ is
$$-2^k\cdot 4 e_2h_{k-2} -2^{k+1}(-2e_1)h_{k-1} -2^{k+2}h_k.$$
Dividing this expression by $-2^{k+2}$ we obtain
$e_2h_{k-2}-e_1h_{k-1}+h_k$, which is the $k$'th coefficient of the
power series
$$(1-e_1t+e_2t)(1+h_1t+h_2t^2+\ldots)=\frac{(1-at)(1-bt)}{(1-at)(1-bt)}=1,$$
hence it is 0. We obtain that substituting (\ref{i22chern}) into the
expression $\mathcal{R}$ is $-4d_\ell\cdot s(\ell+2,\ell+2)$.
Lemma~\ref{adjlemma} implies that
$n_{A_0}=(-2a)(-2b)d_\ell/((-a)(-b))=4d_\ell$ for the germ
$I_{2,2}$. Thus we proved that formula (\ref{III22main}) holds for
stable representatives of $I_{2,2}$ singularities.

\subsection{Restriction to $III_{2,2}$ singularities.} Stable
singularities of type $III_{2,2}$ of relative dimension $\ell$ are
universal unfoldings of the germ $\C^2\to \C^{\ell+2}$
$$(x,y)\mapsto (x^2,y^2,xy,0,\ldots,0).$$
The maximal torus of the maximal compact symmetry group of this germ
is $U(1)^2 \times U(\ell-1)$, and the formal difference of the
representations of this group on the target and on the source is:
$$\rho_1^2 \oplus \rho_1^{'2} \oplus (\rho_1\otimes \rho_1') \oplus \rho_{\ell-1} - (\rho_1 \oplus
\rho_1').$$ Here $\rho_1$ and $\rho_1'$ are the standard
representations of the two $U(1)$ factors, and $\rho_{\ell-1}$ is
the standard representation of $U(\ell-1)$. Therefore, the Chern
classes $c_i$ of the stable representative of $I_{2,2}$ are obtained
by
\begin{equation}\label{iii22chern}
1+c_1t+c_2t^2+\ldots = \frac{(1-2a
t)(1-2bt)(1-(a+b)t)}{(1-at)(1-bt)}\sum_{i=0}^{\ell-1}
d_it^i,\end{equation} where $-a$ and $-b$ are the first Chern
classes of the two $U(1)$ factors, and $d_i$ are the Chern classes
of $U(\ell)$.

This shows that substituting (\ref{iii22chern}) into $\mathcal{R}$
can be obtained by first substituting (\ref{i22chern}) into
$\mathcal{R}$, then plugging in $d_\ell=-(a+b)$. The same holds for
the other terms of (\ref{III22main}) as well, hence the satisfaction
of formula (\ref{III22main}) for substitution (\ref{iii22chern})
follows from the fact that it is satisfied for the substitution
(\ref{i22chern}).

\smallskip

The proof of Theroem \ref{III22} is complete. \end{proof}

\begin{remark} \rm
One can consider applications of the $III_{2,2}A_0$-formula in
enumerative geometry along the lines of Section \ref{4-secants}. The
outcome of such a calculation is then the number (or cohomology
class) of $k$-planes in $\Pn^N$ that have two common points with a
fixed smooth projective variety $V\subset \Pn^N$; one common point
is a transversal intersection, and the other is a singular one, with
singularity $III_{2,2}$.
\end{remark}

\bibliographystyle{plain}
\bibliography{quad}

\end{document}